\documentclass[11pt,leqno]{amsart}

\usepackage{amsmath}
\usepackage{amssymb}
\usepackage{amsthm}
\usepackage{amsfonts}
\usepackage{mathtools}
\usepackage{ascmac}
\usepackage{spalign}
\usepackage{fancybox} 
\usepackage[mathscr]{eucal}
\usepackage{footnpag}
\usepackage{siunitx}
\usepackage{multicol}
\usepackage{bbm}
\usepackage{framed}
\usepackage{here}

\usepackage{graphicx}
\usepackage{xcolor}
\usepackage{wrapfig}
\usepackage{float}

\usepackage{url}
\usepackage{comment}

\theoremstyle{plain}
\newtheorem{theorem}{Theorem}[section]
\newtheorem{proposition}[theorem]{Proposition}
\newtheorem{lemma}[theorem]{Lemma}
\newtheorem{corollary}[theorem]{Corollary}
\theoremstyle{definition}
\newtheorem{definition}[theorem]{Definition}
\theoremstyle{remark}
\newtheorem{remark}[theorem]{Remark}

\title[Average Hitting Times and Recurrence Structures II]
{Average hitting times and recurrence structures II: Cartesian products of powers of cycles and regular graphs}

\author[Miezaki]{Tsuyoshi Miezaki}
\address{Faculty of Science and Engineering, Waseda University, Tokyo 169--8555, Japan}
\email{miezaki@waseda.jp} 

\author[Tamura]{Shunya Tamura*}
\thanks{*Corresponding author}
\address{Okegawa City Okegawa West Junior High School, Saitama, 363-0027, Japan}
\email{shunya.tamura059@gmail.com} 


\makeatletter
\@namedef{subjclassname@2020}{\textup{2020} Mathematics Subject Classification}
\makeatother

\begin{document}
\maketitle
\begin{abstract}
In our previous work \cite{MiezakiTamura2026}, 
we clarified the second-order linear recurrence structures appearing in the average hitting times on the $k$-th power graph $C_N^k$ of the cycle graph. 
In this paper, for a connected $r$-regular graph $G$ on $m$ vertices, 
we investigate the average hitting times of the simple random walk on the Cartesian product graph $C_N^k \square G$. 
By using discrete Fourier analysis in the $C_N^k$ direction and the Laplacian spectral decomposition of $G$, 
we decompose the average hitting time into a component proportional to the average hitting time on $C_N^k$ and correction terms arising from the nonzero Laplacian eigenspaces of $G$. 
For each nonzero Laplacian eigenvalue, we introduce a Chebyshev-type polynomial, and when all of its roots are simple, 
we express the correction term as a finite Green-type sum. 
Furthermore, for two vertices having the same $G$-coordinate, 
we transform this expression into a second-order linear recurrence representation of the form $V_\ell V_{N-\ell}/V_N$. 

When $G$ is a walk-regular graph, the average hitting time between two vertices having the same $G$-coordinate depends only on the Laplacian eigenvalues of $G$ and their multiplicities. 
We also derive formulas for the number of spanning trees and the number of two-component spanning forests of $C_N^k \square G$, 
and give several explicit examples. 
\end{abstract}

\section{Introduction}
\label{sec01}

The average hitting time of a simple random walk on a graph is a fundamental quantity connecting probability theory, 
spectral graph theory, and electrical network theory \cite{DoiEtAl2022, MiezakiTamura2026, NashWilliams1959, Tamura2026, Wu2004}. 
For a finite connected graph $Y$ and two vertices $x, y\in V(Y)$, 
we denote by $h_Y(x, y)$ the expected time for a simple random walk starting from $x$ to reach $y$ for the first time. 
The average hitting time is related to the effective resistance by
\[
h_Y(x, y)+h_Y(y, x)
=2|E(Y)|R_Y(x, y)
\]
\cite{NashWilliams1959}, and is also closely related to the number of spanning trees and the number of two-component spanning forests.

In this paper, we consider the Cartesian product
\[
X=C_N^k\square G
\]
of the $k$-th power graph $C_N^k$ of the cycle graph and a connected $r$-regular graph $G$ on $m$ vertices. 
Here, the vertex set of $C_N^k$ is $\mathbb Z/N\mathbb Z$, and two vertices $x, y$ are adjacent if
\[
x-y\equiv\pm s\pmod N, \qquad 1 \le s \le k. 
\]
Throughout this paper, we assume that $N \ge 2k+1$. 
Then $X$ is a $(2k+r)$-regular graph. 

In our previous work \cite{MiezakiTamura2026}, 
we showed that the correction terms appearing in the average hitting times on $C_N^k$ can be described by second-order linear recurrence relations. 
For example, when $k=2$, we have
\[
h_{C_N^2}(0, \ell)
=\frac{2}{5}\left\{\ell(N-\ell)+\frac{2NF_\ell F_{N-\ell}}{F_N}\right\}
\]
\cite{DoiEtAl2022, MiezakiTamura2026}. 
Here, $F_n$ denotes the $n$-th Fibonacci number. 
The purpose of this paper is to extend such a product structure to the Cartesian product graph $C_N^k\square G$. 

A general framework for expressing Green functions on Cartesian products of regular graphs in terms of the spectra of the factor graphs was given by Ellis \cite{Ellis2003}. 
The main novelty of this paper is that, 
by taking the first factor to be $C_N^k$, 
we describe each spectral correction term by an explicit Chebyshev-type polynomial and explicitly transform its Green-type representation into a product representation involving second-order linear recurrence sequences. 

More precisely, we combine discrete Fourier analysis in the $C_N^k$ direction with the Laplacian spectral decomposition of $G$. 
Let the distinct Laplacian eigenvalues of $G$ be 
\[
0=\nu_0<\nu_1<\cdots<\nu_t,
\]
and let $E_\alpha$ denote the orthogonal projection onto the eigenspace corresponding to $\nu_\alpha$. 
Then, for any $a, b \in V(G)$ and $0 \le \ell \le N-1$, 
we obtain 
\[
h_X((0, a), (\ell, b))
=\frac{2k+r}{2k}h_{C_N^k}(0, \ell)+\mathcal R_{N, k, G}(\ell; a, b).
\]
The first term arises from the constant eigenspace on the $G$-side, 
whereas the second term is a correction term arising from the nonzero Laplacian eigenspaces. 
For a general regular graph, this correction term involves not only the eigenvalues but also the entries $E_\alpha(a, b)$ of the spectral projections. 

For each nonzero Laplacian eigenvalue $\nu$,
we introduce the Chebyshev-type polynomial
\[
D_{k, \nu}(x)
=k+\frac{\nu}{2}-\sum_{s=1}^{k}T_s\left(\frac{x}{2}\right). 
\]
When all of its roots are simple, 
each correction term can be expressed as a finite sum of Green-type terms corresponding to the roots of $D_{k, \nu}(x)$. 

In particular, for two vertices $(0, a)$ and $(\ell, a)$ having the same $G$-coordinate, 
we use the sequences $\{V_n^{(\alpha, c)}\}_{n\ge0}$ determined by second-order linear recurrence relations to express the correction term as a product of the form
\[
\frac{V_\ell^{(\alpha, c)}V_{N-\ell}^{(\alpha, c)}}{V_N^{(\alpha, c)}}.
\]
This generalizes the Fibonacci-type representation appearing in our previous work to a family of recurrence relations corresponding to the nonzero Laplacian eigenvalues of $G$. 

Furthermore,
when $G$ is a walk-regular graph \cite{GodsilMcKay1980}, 
we have 
\[
E_\alpha(a, a)
=\frac{m_\alpha}{m},
\]
where $m_\alpha$ is the multiplicity of $\nu_\alpha$. 
Therefore, the average hitting time between two vertices having the same $G$-coordinate can be described only in terms of the Laplacian eigenvalues of $G$ and their multiplicities. 

Using the same spectral decomposition, 
we also derive a formula for the number of spanning trees of $C_N^k\square G$. 
Furthermore, from the relation between average hitting times and effective resistances, 
we obtain a formula for the number of two-component spanning forests. 
In particular, for two vertices having the same $G$-coordinate, 
the recurrence structures appearing in the average hitting times are also reflected in the number of two-component spanning forests. 

As explicit examples, we consider the cases in which the second factor is a complete graph, a complete bipartite graph, or the Petersen graph. 
Through these examples, we show that the distinct nonzero Laplacian eigenvalues of $G$ and their multiplicities determine the family of recurrence relations appearing in the average hitting times. 

The rest of this paper is organized as follows. 
In Section \ref{sec02}, we prepare basic notation and spectral representations. 
In Section \ref{sec03}, we decompose the average hitting time into a constant component and spectral correction terms. 
In Section \ref{sec04}, we derive the Green-type representation and the product representation involving second-order linear recurrence sequences. 
In Section \ref{sec05}, we give applications to the number of spanning trees and the number of two-component spanning forests. 
In Section \ref{sec06}, we discuss explicit examples, 
and in Section \ref{sec07}, we give concluding remarks and future problems.

\section{Preliminaries}
\label{sec02}

In this section, we prepare the notation and basic facts used throughout this paper. 
Unless otherwise stated, all graphs considered in this paper are finite, simple, and undirected. 
We also assume that $N$ and $k$ are positive integers. 
In what follows, we write $\mathbb Z_N=\mathbb Z/N\mathbb Z$. 

\subsection{The Cartesian product graph}
Let $C_N$ be the cycle graph on $N$ vertices. 
Its $k$-th power graph $C_N^k$ is defined by
\[
V(C_N^k)=\mathbb Z_N, \qquad
E(C_N^k)
=\left\{\{x,y\}\;\Big|\;x-y\equiv\pm s\pmod N,\;1\le s\le k\right\}.
\]
In what follows, we assume that $N \ge 2k+1$. 
Then $C_N^k$ is a $2k$-regular graph. 

Let $G$ be a connected $r$-regular graph on $m$ vertices, and let
\[
V(G)=\{1,2,\ldots,m\}.
\]
In this paper, we consider the Cartesian product graph
\[
X=C_N^k\square G.
\]
Its vertex set is
\[
V(X)=\mathbb Z_N\times V(G).
\]

Two vertices $(x,a)$ and $(y,b)$ are adjacent in $X$ if
\[
a=b, \qquad
x-y\equiv\pm s\pmod N, \qquad
1 \le s \le k,
\]
or
\[
x=y, \qquad
a \sim_G b.
\]
Here, $a \sim_G b$ means that $a$ and $b$ are adjacent in $G$. 

Therefore, $X$ is
\[
d=2k+r
\]
regular, and
\[
|V(X)|=Nm.
\]

\subsection{Random walks and hitting times}
We consider the simple random walk $\{Z_t\}_{t\ge0}$ on $X$. 
At each time, the walker moves from its current position to each adjacent vertex with probability $1/d$. 
For vertices $u, v \in V(X)$, let
\[
T_v=\inf\{t\ge0\mid Z_t=v\}
\]
be the hitting time of $v$, and define
\[
h_X(u,v)=\mathbb E_u[T_v]
\]
as the average hitting time from $u$ to $v$. 
In this paper, we mainly investigate
\[
h_X((0,a),(\ell,b)), \qquad
a, b \in V(G), \qquad
0 \le \ell \le N-1.
\]

\subsection{Fourier analysis on $C_N^k$}
For each $j=0, 1, \ldots, N-1$, put $\theta_j=2\pi j/N$, and define
\[
\chi_j(x)
=\frac1{\sqrt N}e^{ix\theta_j} \qquad (x \in \mathbb Z_N).
\]
Then
\[
\chi_0,\chi_1,\ldots,\chi_{N-1}
\]
form an orthonormal Fourier basis on $\mathbb Z_N$. 

Let $A_k$ be the adjacency matrix of $C_N^k$.
The Fourier diagonalization of circulant adjacency matrices is standard;
see, for example, \cite{BrouwerHaemers2012}.
For the same notation for the power graph $C_N^k$, see also
\cite{MiezakiTamura2026}.
In the present notation, we have
\[
A_k\chi_j
=\mu_j\chi_j,
\]
where
\[
\mu_j
=2\sum_{s=1}^{k}\cos(s\theta_j)
=2\sum_{s=1}^{k}\cos\frac{2\pi js}{N}.
\]

Indeed, by definition, 
\[
\begin{aligned}
(A_k\chi_j)(x)
&=\sum_{s=1}^{k}\left(\chi_j(x+s)+\chi_j(x-s)\right) \\
&=\chi_j(x)\sum_{s=1}^{k}\left(e^{is\theta_j}+e^{-is\theta_j}\right) \\
&=2\sum_{s=1}^{k}\cos(s\theta_j)\chi_j(x).
\end{aligned}
\]

\subsection{Spectral data of $G$}
Let $A_G$ be the adjacency matrix of $G$, and let its Laplacian matrix $L_G$ be
\[
L_G=rI_m-A_G. 
\]
Let the distinct Laplacian eigenvalues of $G$ be
\[
0=\nu_0<\nu_1<\cdots<\nu_t,
\]
and let $m_\alpha$ be the multiplicity of $\nu_\alpha$. 
Since $G$ is connected, 
\[
m_0=1.
\]

Let $E_\alpha$ be the orthogonal projection onto the eigenspace corresponding to $\nu_\alpha$. 
Then
\[
L_G=\sum_{\alpha=0}^{t}\nu_\alpha E_\alpha, \qquad
I_m=\sum_{\alpha=0}^{t}E_\alpha. 
\]

Since the zero eigenspace is spanned by the constant vector,
\[
E_0=\frac1mJ_m. 
\]
Here, $J_m$ is the $m\times m$ matrix whose entries are all equal to $1$. 
Therefore,
\[
E_0(a,b)=\frac{1}{m} \qquad (a, b\in V(G)). 
\]

Moreover, since
\[
A_G=rI_m-L_G,
\]
the eigenvalue of $A_G$ corresponding to $\nu_\alpha$ is
\[
r-\nu_\alpha. 
\]

\subsection{Spectral formula for hitting times}
We prepare the standard spectral representation of average hitting times on regular graphs. 

\begin{lemma}[\cite{Lovasz1993}, Theorem 3.1]
\label{lem03}
Let $Y$ be a connected regular graph on $n$ vertices, 
and let $P_Y$ be the transition probability matrix of the simple random walk. 
Let the eigenvalues of $P_Y$ be
\[
1=\lambda_0,\lambda_1,\ldots,\lambda_{n-1},
\]
and let the corresponding orthonormal eigenfunctions be
\[
\psi_0,\psi_1,\ldots,\psi_{n-1}.
\]
Here, $\psi_0$ is the constant eigenfunction. 
Then
\[
h_Y(u,v)
=n\sum_{q=1}^{n-1}\frac{|\psi_q(v)|^2-\psi_q(u)\overline{\psi_q(v)}}{1-\lambda_q}. 
\]
\end{lemma}

\begin{proof}
For the reader's convenience, we give the proof in the regular case.

Since $Y$ is regular, the stationary distribution of the simple random walk is uniform, and
\[
\pi(v)=\frac1n \qquad (v \in V(Y)). 
\]

Put
\[
\Pi=\frac1nJ_n. 
\]
Then $\Pi$ is the orthogonal projection onto the constant eigenspace. 
Since $P_Y$ is a real symmetric matrix, it admits an orthonormal spectral decomposition, and we can write
\[
P_Y=\Pi+\sum_{q=1}^{n-1}\lambda_q\psi_q\psi_q^\ast. 
\]
Here, $\psi_q^\ast$ denotes the conjugate transpose of $\psi_q$. 

Therefore, 
\[
I_n-P_Y+\Pi
=\Pi+\sum_{q=1}^{n-1}(1-\lambda_q)\psi_q\psi_q^\ast. 
\]
Since $Y$ is connected, we have $\lambda_q\ne1$ for $q\ge1$, and hence
\[
Z
=(I_n-P_Y+\Pi)^{-1}
=\Pi+\sum_{q=1}^{n-1}\frac1{1-\lambda_q}\psi_q\psi_q^\ast.
\]

We use the standard formula for average hitting times
\[
h_Y(u,v)
=\frac{Z(v,v)-Z(u,v)}{\pi(v)}
\]
\cite[Theorem 3.1]{Lovasz1993}.
Since $\pi(v)=1/n$, we obtain
\[
h_Y(u,v)
=n\left\{Z(v,v)-Z(u,v)\right\}.
\]

On the other hand, since all entries of $\Pi$ are equal to $1/n$, 
the terms arising from $\Pi$ cancel in the difference $Z(v,v)-Z(u,v)$. 
Therefore,
\[
\begin{aligned}
Z(v,v)-Z(u,v)
&=\sum_{q=1}^{n-1}\frac{|\psi_q(v)|^2-\psi_q(u)\overline{\psi_q(v)}}{1-\lambda_q}.
\end{aligned}
\]
Hence,
\[
h_Y(u,v)
=n\sum_{q=1}^{n-1}\frac{|\psi_q(v)|^2-\psi_q(u)\overline{\psi_q(v)}}{1-\lambda_q}. 
\]

Moreover, the sum over the eigenfunctions corresponding to each eigenvalue can be expressed in terms of the orthogonal projection onto the corresponding eigenspace. 
Therefore, this sum is independent of the choice of an orthonormal basis within each eigenspace, 
and the same formula holds when a complex orthonormal eigenbasis is used. 
\end{proof}

The spectral behavior of Cartesian product graphs is standard;
see, for example, \cite{BrouwerHaemers2012}.
For the reader's convenience, we recall it in the present notation.

The adjacency matrix of $X=C_N^k\square G$ is
\[
A_X
=A_k \otimes I_m+I_N \otimes A_G. 
\]
Therefore, its spectrum, counted with multiplicities, consists of
\[
\mu_j+r-\nu_\alpha, \qquad
0\le j\le N-1, \qquad
0\le\alpha\le t,
\]
where each pair $(j,\alpha)$ contributes with multiplicity $m_\alpha$. 
Since the transition probability matrix on $X$ is 
\[
P_X
=\frac1{2k+r}A_X,
\]
we write the corresponding eigenvalues as
\[
\xi_{j,\alpha}
=\frac{\mu_j+r-\nu_\alpha}{2k+r}.
\]
Among these, 
\[
\xi_{0,0}=1
\]
is the unique eigenvalue $1$ corresponding to the constant eigenfunction. 

\begin{proposition}
\label{prop04}
For $a, b \in V(G)$ and $0\le\ell\le N-1$, we have
\[
\begin{aligned}
h_X((0,a),(\ell,b))
&=\sum_{j=1}^{N-1}\frac{1-\cos(\ell\theta_j)}{1-\xi_{j,0}} \\
&\quad+m\sum_{\alpha=1}^{t}\sum_{j=0}^{N-1}\frac{E_\alpha(b,b)-E_\alpha(a,b)\cos(\ell\theta_j)}{1-\xi_{j,\alpha}}.
\end{aligned}
\]
\end{proposition}

\begin{proof}
For each $\alpha=0,1,\ldots,t$, let
\[
u_{\alpha,1}, u_{\alpha,2}, \ldots, u_{\alpha,m_\alpha}
\]
be a real orthonormal basis of the eigenspace corresponding to $\nu_\alpha$. 
Then
\[
E_\alpha(a,b)
=\sum_{q=1}^{m_\alpha}u_{\alpha,q}(a)u_{\alpha,q}(b).
\]

The orthonormal eigenfunctions on $X$ are
\[
\Psi_{j,\alpha,q}(x,c)
=\chi_j(x)u_{\alpha,q}(c)
=\frac1{\sqrt N}e^{ix\theta_j}u_{\alpha,q}(c),
\]
and the corresponding eigenvalue is $\xi_{j,\alpha}$. 

First, suppose that $\alpha=0$. 
In this case, 
\[
u_{0,1}
=\frac1{\sqrt m}(1,1,\ldots,1)^t.
\]
Applying Lemma \ref{lem03}, the contribution excluding $(j,\alpha)=(0,0)$, 
which corresponds to the constant eigenfunction, is 
\[
\sum_{j=1}^{N-1}
\frac{1-e^{-i\ell\theta_j}}{1-\xi_{j,0}}. 
\]
By combining the terms corresponding to $j$ and $-j\bmod N$, 
the imaginary parts cancel, and hence this is equal to 
\[
\sum_{j=1}^{N-1}\frac{1-\cos(\ell\theta_j)}{1-\xi_{j, 0}}. 
\]

Next, suppose that $\alpha\ge1$. 
For fixed $j$ and $\alpha$, summing the numerator in Lemma \ref{lem03} over $q=1, \ldots, m_\alpha$, we obtain
\[
\begin{aligned}
&Nm\sum_{q=1}^{m_\alpha}\left\{|\Psi_{j,\alpha,q}(\ell,b)|^2-\Psi_{j,\alpha,q}(0,a)\overline{\Psi_{j,\alpha,q}(\ell,b)}\right\} \\
&=m\left\{E_\alpha(b,b)-e^{-i\ell\theta_j}E_\alpha(a,b)\right\}.
\end{aligned}
\]
Therefore, the contribution corresponding to $\nu_\alpha$ is
\[
m\sum_{j=0}^{N-1}\frac{E_\alpha(b,b)-e^{-i\ell\theta_j}E_\alpha(a,b)}{1-\xi_{j,\alpha}}. 
\]

Here, $E_\alpha(a,b)$ is real, and
\[
\xi_{j,\alpha}
=\xi_{-j\bmod N,\alpha}.
\]
By combining the terms corresponding to $j$ and $-j\bmod N$, this contribution is equal to
\[
m\sum_{j=0}^{N-1}\frac{E_\alpha(b,b)-E_\alpha(a,b)\cos(\ell\theta_j)}{1-\xi_{j,\alpha}}. 
\]

Combining all the contributions proves the assertion. 
\end{proof}

By Proposition \ref{prop04}, the average hitting time is decomposed into a component arising from the constant eigenspace on the $G$-side and components arising from the nonzero Laplacian eigenspaces. 
In the next section, we express the former in terms of the average hitting time on $C_N^k$, 
and describe the latter by Chebyshev-type polynomials.

\section{Spectral decomposition}
\label{sec03}

In this section, we decompose the spectral representation in Proposition \ref{prop04} into the component arising from the constant eigenspace on the $G$-side and the components arising from the nonzero Laplacian eigenspaces. 

Throughout this section, $\theta_j$ and $\mu_j$ are defined as in Section 2. 
Namely, 
\[
\theta_j=\frac{2\pi j}{N}, \qquad
\mu_j=2\sum_{s=1}^{k}\cos(s\theta_j).
\]

In the right-hand side of Proposition \ref{prop04}, 
we denote the component arising from the constant eigenspace on the $G$-side by 
\[
\mathcal S_0
=\sum_{j=1}^{N-1}\frac{1-\cos(\ell\theta_j)}{1-\xi_{j,0}}. 
\]

For each $\alpha=1, \ldots, t$, 
we also denote the component arising from the eigenspace corresponding to $\nu_\alpha$ by
\[
\mathcal S_\alpha
=m\sum_{j=0}^{N-1}\frac{E_\alpha(b,b)-E_\alpha(a,b)\cos(\ell\theta_j)}{1-\xi_{j,\alpha}}. 
\]

Then, by Proposition \ref{prop04}, 
\[
h_X((0,a),(\ell,b))
=\mathcal S_0+\sum_{\alpha=1}^{t}\mathcal S_\alpha.
\]

\begin{proposition}[\cite{MiezakiTamura2026}, Proposition 2.4]
\label{propCyclePowerHitting}
For $0\le \ell\le N-1$, we have
\[
h_{C_N^k}(0,\ell)
=
2k\sum_{j=1}^{N-1}
\frac{1-\cos(\ell\theta_j)}{2k-\mu_j}.
\]
\end{proposition}

\begin{proof}
For the reader's convenience, we recall the proof from the spectral formula.
The transition matrix of the simple random walk on $C_N^k$ has eigenvalues
\[
\frac{\mu_j}{2k}
\qquad
(j=0,1,\ldots,N-1).
\]
Applying Lemma \ref{lem03} to the regular graph $C_N^k$, and excluding the
constant eigenfunction corresponding to $j=0$, we obtain
\[
\begin{aligned}
h_{C_N^k}(0,\ell)
&=
\sum_{j=1}^{N-1}
\frac{1-e^{-i\ell\theta_j}}{1-\mu_j/(2k)}.
\end{aligned}
\]
By combining the terms corresponding to $j$ and $-j\bmod N$, the imaginary
parts cancel, and hence
\[
\begin{aligned}
h_{C_N^k}(0,\ell)
&=
\sum_{j=1}^{N-1}
\frac{1-\cos(\ell\theta_j)}{1-\mu_j/(2k)}
\\
&=
2k\sum_{j=1}^{N-1}
\frac{1-\cos(\ell\theta_j)}{2k-\mu_j}.
\end{aligned}
\]
\end{proof}

\begin{proposition}
\label{prop05}
\[
\mathcal S_0
=\frac{2k+r}{2k}h_{C_N^k}(0,\ell).
\]
\end{proposition}

\begin{proof}
Since
\[
\xi_{j,0}
=\frac{\mu_j+r}{2k+r},
\]
we have
\[
1-\xi_{j,0}
=\frac{2k-\mu_j}{2k+r}.
\]
Therefore, by Proposition \ref{prop04}, 
\[
\mathcal S_0
=(2k+r)\sum_{j=1}^{N-1}\frac{1-\cos(\ell\theta_j)}{2k-\mu_j}.
\]

On the other hand, by Proposition \ref{propCyclePowerHitting}, 
\[
h_{C_N^k}(0,\ell)
=2k\sum_{j=1}^{N-1}\frac{1-\cos(\ell\theta_j)}{2k-\mu_j}. 
\]
Combining these two identities proves the assertion. 
\end{proof}

Next, we consider the contributions from the nonzero Laplacian eigenspaces. 
For $\nu>0$, we define the Chebyshev-type polynomial
\[
D_{k,\nu}(x)
=k+\frac{\nu}{2}-\sum_{s=1}^{k}T_s\left(\frac{x}{2}\right).
\]
Here, $T_s(x)$ denotes the Chebyshev polynomial of the first kind. 
By the identity
\[
T_s(\cos\theta)=\cos(s\theta), 
\]
we have
\[
2D_{k,\nu}(2\cos\theta_j)
=2k+\nu-\mu_j.
\]

\begin{proposition}
\label{prop06}
For each $\alpha=1, \ldots, t$, 
\[
\mathcal S_\alpha
=\frac{m(2k+r)}{2}\sum_{j=0}^{N-1}\frac{E_\alpha(b,b)-E_\alpha(a,b)\cos(\ell\theta_j)}{D_{k,\nu_\alpha}(2\cos\theta_j)}.
\]
\end{proposition}

\begin{proof}
Since
\[
\xi_{j,\alpha}
=\frac{\mu_j+r-\nu_\alpha}{2k+r},
\]
we have
\[
1-\xi_{j,\alpha}
=\frac{2k+\nu_\alpha-\mu_j}{2k+r}.
\]
Therefore, by Proposition \ref{prop04}, 
\[
\begin{aligned}
\mathcal S_\alpha
&=m(2k+r)\sum_{j=0}^{N-1}\frac{E_\alpha(b,b)-E_\alpha(a,b)\cos(\ell\theta_j)}{2k+\nu_\alpha-\mu_j}.
\end{aligned}
\]
Using
\[
2D_{k,\nu_\alpha}(2\cos\theta_j)
=2k+\nu_\alpha-\mu_j,
\]
we obtain the assertion. 
\end{proof}

Substituting Propositions \ref{prop05} and \ref{prop06} into
\[
h_X((0,a),(\ell,b))
=\mathcal S_0+\sum_{\alpha=1}^{t}\mathcal S_\alpha,
\]
we obtain the following result. 

\begin{corollary}
\label{cor07}
For $a,b\in V(G)$ and $0\le\ell\le N-1$,
\[
\begin{aligned}
h_X((0,a),(\ell,b))
&=\frac{2k+r}{2k}h_{C_N^k}(0,\ell) \\
&\quad+\frac{m(2k+r)}{2}\sum_{\alpha=1}^{t}\sum_{j=0}^{N-1}\frac{E_\alpha(b,b)-E_\alpha(a,b)\cos(\ell\theta_j)}{D_{k,\nu_\alpha}(2\cos\theta_j)}. 
\end{aligned}
\]
\end{corollary}

Corollary \ref{cor07} shows that the average hitting time is decomposed into a component proportional to the average hitting time on $C_N^k$ and spectral correction terms arising from the nonzero Laplacian eigenspaces of $G$. 

Each correction term is described by the corresponding Laplacian eigenvalue $\nu_\alpha$ and the Chebyshev-type polynomial
\[
D_{k,\nu_\alpha}(x)
=k+\frac{\nu_\alpha}{2}-\sum_{s=1}^{k}T_s\left(\frac{x}{2}\right).
\]

Thus, the average hitting time on $C_N^k\square G$ is decomposed into a component arising from $C_N^k$ and spectral correction terms arising from the nonzero Laplacian eigenspaces of $G$. 
In the next section, we transform each correction term into a Green-type representation in terms of the roots of $D_{k,\nu_\alpha}(x)$,
and then derive a product representation involving second-order linear recurrence sequences.

\section{Explicit formulas for hitting times}
\label{sec04}
In this section, we transform the spectral correction terms in Corollary \ref{cor07} and derive a Green-type representation
and a product representation involving second-order linear recurrence sequences. 

\subsection{Partial fraction decomposition}
For each $\alpha=1,\ldots,t$, consider
\[
D_{k,\nu_\alpha}(x)
=k+\frac{\nu_\alpha}{2}-\sum_{s=1}^{k}T_s\left(\frac{x}{2}\right).
\]
The leading term of the Chebyshev polynomial $T_k(x)$ of the first kind is $2^{k-1}x^k$. 
Therefore, the leading term of
\[
T_k\left(\frac{x}{2}\right)
\]
is $x^k/2$. 
On the other hand, for $s<k$, the degree of $T_s(x/2)$ is $s<k$. 
Hence, $D_{k,\nu_\alpha}(x)$ is a polynomial of degree $k$ with leading coefficient $-1/2$.

Therefore, over the complex numbers, 
$D_{k,\nu_\alpha}(x)$ has $k$ roots counted with multiplicity. 
In what follows, we assume that all of these roots are simple, 
and denote them by $\phi_{\alpha,1}, \ldots, \phi_{\alpha,k}$. 

\begin{remark}
\label{remSimpleRoots}
The simplicity assumption on the roots of $D_{k,\nu_\alpha}(x)$ is used only
to obtain the partial fraction decomposition in Proposition \ref{prop09}
in the simple form below.
The spectral formula in Corollary \ref{cor07} and the spanning tree formula
in Proposition \ref{prop19} do not require this assumption.
When multiple roots occur, analogous formulas can be obtained by using
higher-order partial fractions.
In this paper, we restrict ourselves to the case where all roots are simple.
\end{remark}

We first examine the locations of these roots. 

\begin{lemma}
\label{lem09}
Let $\nu>0$. 
Then
\[
D_{k,\nu}(x)>0 \qquad (-2 \le x \le2).
\]
Therefore, $D_{k,\nu}(x)$ has no roots in the interval $[-2,2]$. 
\end{lemma}

\begin{proof}
Writing $x=2\cos\theta$, we have 
\[
\begin{aligned}
D_{k,\nu}(2\cos\theta)
&=k+\frac{\nu}{2}-\sum_{s=1}^{k}\cos(s\theta) \\
&=\frac{\nu}{2}+\sum_{s=1}^{k}\left(1-\cos(s\theta)\right)
>0. 
\end{aligned}
\]
\end{proof}

\begin{proposition}
\label{prop09}
For each $\alpha=1, \ldots, t$, 
\[
\frac1{D_{k,\nu_\alpha}(x)}
=\sum_{c=1}^{k}\frac1{D_{k,\nu_\alpha}'(\phi_{\alpha,c})}\frac1{x-\phi_{\alpha,c}}. 
\]
\end{proposition}

\begin{proof}
By the partial fraction decomposition for simple roots, we can write 
\[
\frac1{D_{k,\nu_\alpha}(x)}
=\sum_{c=1}^{k}\frac{A_{\alpha,c}}{x-\phi_{\alpha,c}}.
\]
Multiplying both sides by $x-\phi_{\alpha,c}$ and letting $x\to\phi_{\alpha,c}$, 
we obtain
\[
A_{\alpha,c}
=\frac1{D_{k,\nu_\alpha}'(\phi_{\alpha,c})}.
\]
\end{proof}

\subsection{Green-type representation}
For each $\alpha=1, \ldots, t$ and $c=1, \ldots, k$, consider the equation
\[
z+z^{-1}
=\phi_{\alpha,c}.
\]
This equation is equivalent to
\[
z^2-\phi_{\alpha,c}z+1=0. 
\]
Therefore, it has two solutions over the complex numbers, 
and the product of these solutions is $1$. 
In particular, neither solution is zero. 

We choose one of these two solutions and denote it by $\rho_{\alpha,c} \in \mathbb C$. 
Then
\[
\rho_{\alpha,c}+\rho_{\alpha,c}^{-1}
=\phi_{\alpha,c},
\]
and the other solution is $\rho_{\alpha,c}^{-1}$. 

Furthermore, by Lemma \ref{lem09}, 
\[
\phi_{\alpha,c}\notin[-2, 2]. 
\]
It follows that
\[
\rho_{\alpha,c} \ne \rho_{\alpha,c}^{-1}, \qquad
\rho_{\alpha,c}^{N} \ne 1. 
\]

Indeed, suppose that
\[
\rho_{\alpha,c}=\rho_{\alpha,c}^{-1}. 
\]
Then
\[
\rho_{\alpha,c}^{2}=1,
\]
and hence
\[
\rho_{\alpha,c}=\pm1. 
\]
Therefore,
\[
\phi_{\alpha,c}
=\rho_{\alpha,c}+\rho_{\alpha,c}^{-1}
=\pm2,
\]
which contradicts $\phi_{\alpha,c}\notin[-2,2]$. 

Next, suppose that 
\[
\rho_{\alpha,c}^{N}=1. 
\]
Then $\rho_{\alpha,c}$ lies on the unit circle. 
Therefore, for some real number $\vartheta$, we can write
\[
\rho_{\alpha,c}=e^{i\vartheta}. 
\]
In this case,
\[
\begin{aligned}
\phi_{\alpha,c}
&=\rho_{\alpha,c}+\rho_{\alpha,c}^{-1} \\
&=e^{i\vartheta}+e^{-i\vartheta} \\
&=2\cos\vartheta \in [-2,2], 
\end{aligned}
\]
which again contradicts $\phi_{\alpha,c}\notin[-2,2]$. 

Thus,
\[
\rho_{\alpha,c}-\rho_{\alpha,c}^{-1} \ne 0, \qquad
\rho_{\alpha,c}^{N}-1 \ne 0,
\]
and none of the denominators appearing in the Green-type representation below is zero.

\begin{lemma}
\label{lem10}
Let $\phi=\rho+\rho^{-1}$, and suppose that 
\[
\rho \ne 0, \qquad
\rho \ne \rho^{-1}, \qquad
\rho^N \ne 1. 
\]
Then, for $0 \le \ell \le N-1$, 
\[
\sum_{j=0}^{N-1}\frac{e^{i\ell\theta_j}}{\phi-2\cos\theta_j}
=\frac{N\left(\rho^\ell+\rho^{N-\ell}\right)}{(\rho-\rho^{-1})(\rho^N-1)}. 
\]
\end{lemma}

\begin{proof}
Put $\omega=e^{2\pi i/N}$.
Then
\[
\frac1{\phi-2\cos\theta_j}
=\frac1{\rho-\rho^{-1}}\left(\frac1{1-\rho^{-1}\omega^j}-\frac1{1-\rho\omega^j}\right).
\]

Moreover, when $a^N \ne 1$, 
\[
\sum_{j=0}^{N-1}\frac{\omega^{j\ell}}{1-a\omega^j}
=
\begin{cases}
\dfrac{N}{1-a^N},
&\ell=0,\\[3mm]
\dfrac{Na^{N-\ell}}{1-a^N},
&1\le\ell\le N-1.
\end{cases}
\]
Indeed, since
\[
\frac{1-a^N}{1-a\omega^j}
=\sum_{p=0}^{N-1}a^p\omega^{jp},
\]
we have
\[
\begin{aligned}
(1-a^N)\sum_{j=0}^{N-1}\frac{\omega^{j\ell}}{1-a\omega^j}
&=\sum_{p=0}^{N-1}a^p\sum_{j=0}^{N-1}\omega^{j(p+\ell)}.
\end{aligned}
\]
Here,
\[
\sum_{j=0}^{N-1}\omega^{j(p+\ell)}
=
\begin{cases}
N,
&p+\ell\equiv0\pmod N,\\
0,
&p+\ell\not\equiv0\pmod N.
\end{cases}
\]
Therefore, when $\ell=0$, only the term with $p=0$ remains, whereas when $1\le\ell\le N-1$,
only the term with $p=N-\ell$ remains. 

Applying this identity to $a=\rho^{-1}$ and $a=\rho$ proves the assertion. 
\end{proof}

For each $\alpha=1,\ldots,t$ and $c=1,\ldots,k$, define
\[
\mathcal G_{\alpha,c}(\ell)
=\frac{N\left(\rho_{\alpha,c}^{\ell}+\rho_{\alpha,c}^{N-\ell}\right)}{\left(\rho_{\alpha,c}-\rho_{\alpha,c}^{-1}\right)\left(\rho_{\alpha,c}^{N}-1\right)}. 
\]
This value remains unchanged when $\rho_{\alpha,c}$ is replaced by $\rho_{\alpha,c}^{-1}$. 

\begin{proposition}
\label{prop11}
For each $\alpha=1, \ldots, t$ and $0 \le \ell \le N-1$, 
\[
\sum_{j=0}^{N-1}\frac{\cos(\ell\theta_j)}{D_{k,\nu_\alpha}(2\cos\theta_j)}
=-\sum_{c=1}^{k}\frac{\mathcal G_{\alpha,c}(\ell)}{D_{k,\nu_\alpha}'(\phi_{\alpha,c})}.
\]
\end{proposition}

\begin{proof}
By Proposition \ref{prop09},
\[
\begin{aligned}
\sum_{j=0}^{N-1}\frac{\cos(\ell\theta_j)}{D_{k,\nu_\alpha}(2\cos\theta_j)}
&=-\sum_{c=1}^{k}\frac1{D_{k,\nu_\alpha}'(\phi_{\alpha,c})}\sum_{j=0}^{N-1}\frac{\cos(\ell\theta_j)}{\phi_{\alpha,c}-2\cos\theta_j}.
\end{aligned}
\]

When $\ell=0$, Lemma \ref{lem10} gives
\[
\sum_{j=0}^{N-1}
\frac1{\phi_{\alpha,c}-2\cos\theta_j}
=\mathcal G_{\alpha,c}(0).
\]

When $1\le\ell\le N-1$, using
\[
\cos(\ell\theta_j)
=\frac{e^{i\ell\theta_j}+e^{-i\ell\theta_j}}{2}
\]
and
\[
e^{-i\ell\theta_j}
=e^{i(N-\ell)\theta_j},
\]
we obtain
\[
\begin{aligned}
\sum_{j=0}^{N-1}\frac{\cos(\ell\theta_j)}{\phi_{\alpha,c}-2\cos\theta_j}
&=\frac12\left\{\mathcal G_{\alpha,c}(\ell)+\mathcal G_{\alpha,c}(N-\ell)\right\} \\
&=\mathcal G_{\alpha,c}(\ell).
\end{aligned}
\]
Therefore, the assertion follows. 
\end{proof}

We write the correction term in Corollary \ref{cor07} as
\[
\begin{aligned}
\mathcal R_{N,k,G}(\ell;a,b)
&=\frac{m(2k+r)}{2}\sum_{\alpha=1}^{t}\sum_{j=0}^{N-1}\frac{E_\alpha(b,b)-E_\alpha(a,b)\cos(\ell\theta_j)}{D_{k,\nu_\alpha}(2\cos\theta_j)}. 
\end{aligned}
\]

\begin{proposition}
\label{prop12}
For $a,b\in V(G)$ and $0\le\ell\le N-1$, we have
\[
\begin{aligned}
\mathcal R_{N,k,G}(\ell;a,b)
&=-\frac{m(2k+r)}{2}\sum_{\alpha=1}^{t}\sum_{c=1}^{k}\frac{E_\alpha(b,b)\mathcal G_{\alpha,c}(0)-E_\alpha(a,b)\mathcal G_{\alpha,c}(\ell)}{D_{k,\nu_\alpha}'(\phi_{\alpha,c})}. 
\end{aligned}
\]
\end{proposition}

\begin{proof}
By the definition of the correction term, 
\[
\begin{aligned}
\mathcal R_{N,k,G}(\ell;a,b)
&=\frac{m(2k+r)}{2}\sum_{\alpha=1}^{t}E_\alpha(b,b)\sum_{j=0}^{N-1}\frac1{D_{k,\nu_\alpha}(2\cos\theta_j)} \\
&\quad-\frac{m(2k+r)}{2}\sum_{\alpha=1}^{t}E_\alpha(a,b)\sum_{j=0}^{N-1}\frac{\cos(\ell\theta_j)}{D_{k,\nu_\alpha}(2\cos\theta_j)}.
\end{aligned}
\]

Taking $\ell=0$ in Proposition \ref{prop11}, we obtain
\[
\sum_{j=0}^{N-1}\frac1{D_{k,\nu_\alpha}(2\cos\theta_j)}
=-\sum_{c=1}^{k}\frac{\mathcal G_{\alpha,c}(0)}{D_{k,\nu_\alpha}'(\phi_{\alpha,c})}.
\]
Moreover, for a general $\ell$, 
\[
\sum_{j=0}^{N-1}\frac{\cos(\ell\theta_j)}{D_{k,\nu_\alpha}(2\cos\theta_j)}
=-\sum_{c=1}^{k}\frac{\mathcal G_{\alpha,c}(\ell)}{D_{k,\nu_\alpha}'(\phi_{\alpha,c})}.
\]

Substituting these identities into the expression for the correction term, we obtain
\[
\begin{aligned}
\mathcal R_{N,k,G}(\ell;a,b)
&=-\frac{m(2k+r)}{2}\sum_{\alpha=1}^{t}\sum_{c=1}^{k}\frac{E_\alpha(b,b)\mathcal G_{\alpha,c}(0)}{D_{k,\nu_\alpha}'(\phi_{\alpha,c})} \\
&\quad+\frac{m(2k+r)}{2}\sum_{\alpha=1}^{t}\sum_{c=1}^{k}\frac{E_\alpha(a,b)\mathcal G_{\alpha,c}(\ell)}{D_{k,\nu_\alpha}'(\phi_{\alpha,c})}. 
\end{aligned}
\]
Combining these terms proves the assertion. 
\end{proof}

Therefore, the correction term for any two vertices can be expressed as a finite sum of Green-type terms corresponding to the nonzero Laplacian eigenspaces of $G$.

\subsection{Recurrence representation}

We now consider two vertices $(0,a)$ and $(\ell,a)$ having the same $G$-coordinate. 
Since each $\rho_{\alpha,c}$ is a nonzero complex number, 
it has a square root over the complex numbers. 
Therefore, for each $\alpha=1, \ldots, t$ and $c=1, \ldots, k$, we choose $\eta_{\alpha,c} \in \mathbb C$ satisfying
\[
\rho_{\alpha,c}
=\eta_{\alpha,c}^{\,2},
\]
and define
\[
\gamma_{\alpha,c}
=\eta_{\alpha,c}+\eta_{\alpha,c}^{-1}.
\]
Then
\[
\begin{aligned}
\gamma_{\alpha,c}^{\,2}
&=\eta_{\alpha,c}^{\,2}+2+\eta_{\alpha,c}^{-2} \\
&=\rho_{\alpha,c}+\rho_{\alpha,c}^{-1}+2 \\
&=\phi_{\alpha,c}+2.
\end{aligned}
\]

By Lemma \ref{lem09}, $\phi_{\alpha,c}\ne-2$, and hence 
\[
\gamma_{\alpha,c}\ne0.
\]
Moreover, since $\rho_{\alpha,c}\ne1$, we have
\[
\eta_{\alpha,c}\ne\pm1, \qquad
\eta_{\alpha,c}-\eta_{\alpha,c}^{-1}
\ne 0.
\]

\begin{definition}
\label{def13}
For each $\alpha=1, \ldots, t$ and $c=1, \ldots, k$, define the sequence $\{V_n^{(\alpha,c)}\}_{n\ge0}$ by
\[
V_0^{(\alpha,c)}=0, \qquad
V_1^{(\alpha,c)}=1, \qquad
V_{n+2}^{(\alpha,c)}
=\gamma_{\alpha,c}V_{n+1}^{(\alpha,c)}-V_n^{(\alpha,c)}.
\]
\end{definition}

\begin{lemma}
\label{lem14}
For each $n\ge0$, we have
\[
V_n^{(\alpha,c)}
=\frac{\eta_{\alpha,c}^{\,n}-\eta_{\alpha,c}^{-n}}{\eta_{\alpha,c}-\eta_{\alpha,c}^{-1}}. 
\]
\end{lemma}

\begin{proof}
Let the right-hand side be denoted by $W_n$. 
Then
\[
W_0=0, \qquad
W_1=1, \qquad
W_{n+2}=\gamma_{\alpha,c}W_{n+1}-W_n.
\]
Therefore, by the uniqueness of the sequence satisfying the recurrence relation, 
\[
W_n=V_n^{(\alpha,c)}.
\]
\end{proof}

Furthermore, since $\rho_{\alpha,c}^{N} \neq 1$, 
we have $\eta_{\alpha,c}^{\,2N} \neq 1$, 
and hence $V_N^{(\alpha,c)} \neq 0$. 
Moreover, by the assumption that all roots are simple, 
\[
D_{k,\nu_\alpha}'(\phi_{\alpha,c})\ne0. 
\]
Therefore, none of the denominators appearing below is zero. 

\begin{proposition}
\label{prop15}
For each $\alpha=1, \ldots, t$, $c=1, \ldots, k$, and $0 \le \ell \le N-1$, we have
\[
\mathcal G_{\alpha,c}(0)-\mathcal G_{\alpha,c}(\ell)
=\frac{N}{\gamma_{\alpha,c}}\frac{V_\ell^{(\alpha,c)}V_{N-\ell}^{(\alpha,c)}}{V_N^{(\alpha,c)}}.
\]
\end{proposition}

\begin{proof}
By definition,
\[
\begin{aligned}
\mathcal G_{\alpha,c}(0)-\mathcal G_{\alpha,c}(\ell)
&=\frac{N\left(1-\rho_{\alpha,c}^{\ell}\right)\left(1-\rho_{\alpha,c}^{N-\ell}\right)}{\left(\rho_{\alpha,c}-\rho_{\alpha,c}^{-1}\right)\left(
\rho_{\alpha,c}^{N}-1\right)}.
\end{aligned}
\]
Since $\rho_{\alpha,c}=\eta_{\alpha,c}^{\,2}$, we have
\[
1-\rho_{\alpha,c}^{\,q}
=-\eta_{\alpha,c}^{\,q}\left(\eta_{\alpha,c}^{\,q}-\eta_{\alpha,c}^{-q}\right).
\]
Moreover,
\[
\rho_{\alpha,c}-\rho_{\alpha,c}^{-1}
=\gamma_{\alpha,c}\left(\eta_{\alpha,c}-\eta_{\alpha,c}^{-1}\right),
\]
and
\[
\rho_{\alpha,c}^{N}-1
=\eta_{\alpha,c}^{\,N}\left(\eta_{\alpha,c}^{\,N}-\eta_{\alpha,c}^{-N}\right).
\]
Using these identities together with Lemma \ref{lem14}, we obtain the assertion. 
\end{proof}

\subsection{The main recurrence formula}
The following theorem is the main formula of this paper. 
It extends the representation of the form
\[
\frac{V_\ell V_{N-\ell}}{V_N}
\]
appearing in our previous work \cite{MiezakiTamura2026} to a general connected regular graph $G$. 

\begin{theorem}
\label{thm16}
For each $\alpha=1,\ldots,t$, assume that all roots of $D_{k,\nu_\alpha}(x)$ are simple.
Then, for $a\in V(G)$ and $0\le\ell\le N-1$, we have 
\[
\begin{aligned}
h_X((0,a),(\ell,a))
&=\frac{2k+r}{2k}h_{C_N^k}(0,\ell) \\
&\quad-\frac{Nm(2k+r)}{2}\sum_{\alpha=1}^{t}E_\alpha(a,a)\sum_{c=1}^{k}\frac1{\gamma_{\alpha,c}D_{k,\nu_\alpha}'(\phi_{\alpha,c})}
\frac{V_\ell^{(\alpha,c)}V_{N-\ell}^{(\alpha,c)}}{V_N^{(\alpha,c)}}.
\end{aligned}
\]
\end{theorem}

\begin{proof}
Setting $b=a$ in Proposition \ref{prop12}, we obtain
\[
\begin{aligned}
\mathcal R_{N,k,G}(\ell;a,a)
&=-\frac{m(2k+r)}{2}\sum_{\alpha=1}^{t}E_\alpha(a,a)\sum_{c=1}^{k}\frac{\mathcal G_{\alpha,c}(0)-\mathcal G_{\alpha,c}(\ell)}{D_{k,\nu_\alpha}'(\phi_{\alpha,c})}.
\end{aligned}
\]
Substituting Proposition \ref{prop15} into this expression and using Corollary \ref{cor07}, 
we obtain the assertion. 
\end{proof}

\begin{remark}
The right-hand side of Theorem \ref{thm16} is independent of the choices of $\rho_{\alpha,c}$ and $\eta_{\alpha,c}$. 

First, consider replacing $\eta_{\alpha,c}$ by $-\eta_{\alpha,c}$. 
Then
\[
\gamma_{\alpha,c}
=\eta_{\alpha,c}+\eta_{\alpha,c}^{-1}
\]
is replaced by $-\gamma_{\alpha,c}$. 

Moreover, by Lemma \ref{lem14}, 
\[
V_n^{(\alpha,c)} \longmapsto (-1)^{n-1}V_n^{(\alpha,c)}. 
\]
Therefore,
\[
\frac{V_\ell^{(\alpha,c)}V_{N-\ell}^{(\alpha,c)}}{V_N^{(\alpha,c)}}
\]
is multiplied by $-1$. 
On the other hand, $1/\gamma_{\alpha,c}$ is also multiplied by $-1$. 
Hence,
\[
\frac1{\gamma_{\alpha,c}}\frac{V_\ell^{(\alpha,c)}V_{N-\ell}^{(\alpha,c)}}{V_N^{(\alpha,c)}}
\]
remains unchanged. 

Next, consider replacing $\rho_{\alpha,c}$ by $\rho_{\alpha,c}^{-1}$. 
In this case, we may choose $\eta_{\alpha,c}^{-1}$ instead of $\eta_{\alpha,c}$.
Then
\[
\gamma_{\alpha,c}
=\eta_{\alpha,c}+\eta_{\alpha,c}^{-1}
\]
remains unchanged. 
Moreover, both the numerator and the denominator in Lemma \ref{lem14} change sign, and hence $V_n^{(\alpha,c)}$ also remains unchanged. 

Therefore, the right-hand side of Theorem \ref{thm16} is independent of the choices of $\rho_{\alpha,c}$ and $\eta_{\alpha,c}$. 
\end{remark}

\begin{remark}
\label{remRealValue}
In general, the terms in Theorem \ref{thm16} involve the complex numbers
\[
\phi_{\alpha,c}, \quad
\rho_{\alpha,c}, \quad
\eta_{\alpha,c}, \quad
\gamma_{\alpha,c}.
\]
However, the entire right-hand side is real. 

Since $D_{k,\nu_\alpha}(x)$ is a polynomial with real coefficients, 
its nonreal roots occur in complex conjugate pairs. 
That is, if $\phi_{\alpha,c}$ is a nonreal root, 
then its complex conjugate $\overline{\phi_{\alpha,c}}$ is also a root. 

If the corresponding $\rho$, $\eta$, and $\gamma$ are chosen compatibly with complex conjugation, 
then the terms corresponding to two complex conjugate roots are also complex conjugates of each other. 
Therefore, the sum of these two terms is real. 

On the other hand, for a real root, 
the left-hand side of Proposition \ref{prop15}, 
\[
\mathcal G_{\alpha,c}(0)-\mathcal G_{\alpha,c}(\ell),
\]
and $D_{k,\nu_\alpha}'(\phi_{\alpha,c})$ are real, 
and hence the corresponding term is also real. 

Therefore, after summing over all roots, 
the right-hand side of Theorem \ref{thm16} is real. 
\end{remark}

\subsection{The walk-regular case}
Theorem \ref{thm16} involves the diagonal entries $E_\alpha(a,a)$ of the spectral projections. 
In general, these entries depend on the vertex $a$. 
However, when $G$ is a walk-regular graph, 
they are determined only by the eigenvalue multiplicities \cite{GodsilMcKay1980}. 

\begin{definition}[\cite{GodsilMcKay1980}]
\label{defWalkRegular}
A graph $G$ is said to be walk-regular if, for every nonnegative integer $q$, 
\[
(A_G^q)(a,a)
\]
is independent of the vertex $a$. 
\end{definition}

\begin{lemma}[\cite{GodsilMcKay1980}, Theorem 4.1]
\label{lemWalkRegular}
The graph $G$ is walk-regular if and only if, 
for each $\alpha=0,1,\ldots,t$, 
\[
E_\alpha(a,a)
=\frac{m_\alpha}{m} \qquad (a \in V(G)). 
\]
\end{lemma}

\begin{proof}
For the reader's convenience, we give the proof in the present Laplacian notation. 
Since
\[
A_G
=\sum_{\alpha=0}^{t}(r-\nu_\alpha)E_\alpha,
\]
we have
\[
(A_G^q)(a,a)
=\sum_{\alpha=0}^{t}(r-\nu_\alpha)^qE_\alpha(a,a).
\]
Therefore, if each $E_\alpha(a,a)$ is independent of the vertex,
then $G$ is walk-regular. 

Conversely, suppose that $G$ is walk-regular. 
Using the above identity for $q=0, 1, \ldots, t$ and the fact that $r-\nu_0, \ldots, r-\nu_t$ are mutually distinct, 
the nonsingularity of the Vandermonde matrix implies that each $E_\alpha(a,a)$ is independent of the vertex $a$. 
Furthermore, 
\[
\sum_{a\in V(G)}E_\alpha(a,a)
=\operatorname{tr}(E_\alpha)
=m_\alpha.
\]
Since $E_\alpha(a,a)$ is independent of $a$ and $G$ has $m$ vertices, we obtain
\[
E_\alpha(a,a)=\frac{m_\alpha}{m}.
\]
This proves the assertion. 
\end{proof}

By Lemma \ref{lemWalkRegular}, 
\[
E_\alpha(a,a)=\frac{m_\alpha}{m}.
\]
Substituting this identity into Theorem \ref{thm16}, 
we obtain the following result. 

\begin{corollary}
\label{corWalkRegular}
For each $\alpha=1,\ldots,t$, assume that all roots of $D_{k,\nu_\alpha}(x)$ are simple. 
Furthermore, suppose that $G$ is a walk-regular graph. 
Then, for any $a\in V(G)$ and $0 \le \ell \le N-1$, 
we have
\[
\begin{aligned}
h_X((0,a),(\ell,a))
&=\frac{2k+r}{2k}h_{C_N^k}(0,\ell) \\
&\quad-\frac{N(2k+r)}{2}\sum_{\alpha=1}^{t}m_\alpha\sum_{c=1}^{k}\frac1{\gamma_{\alpha,c}D_{k,\nu_\alpha}'(\phi_{\alpha,c})}\frac{V_\ell^{(\alpha,c)}V_{N-\ell}^{(\alpha,c)}}{V_N^{(\alpha,c)}}.
\end{aligned}
\]
In particular, the right-hand side is independent of the vertex $a$. 
\end{corollary}

Corollary \ref{corWalkRegular} shows that the average hitting time between two vertices having the same $G$-coordinate
can be described only in terms of the Laplacian eigenvalues of $G$ and their multiplicities. 
Therefore, it is not necessary to determine the individual eigenvectors. 

\begin{remark}
\label{rem18}
For any two vertices, the Green-type representation in Proposition \ref{prop12} is obtained. 
In particular, when the two vertices have the same $G$-coordinate, 
the Green-type terms appear in the difference
\[
\mathcal G_{\alpha,c}(0)-\mathcal G_{\alpha,c}(\ell).
\]
Therefore, they can be transformed into the second-order linear recurrence representation in Theorem \ref{thm16}. 
\end{remark}

\section{Combinatorial applications}
\label{sec05}

In this section, using the results obtained in the preceding sections, 
we determine the number of spanning trees and the number of two-component spanning forests of the Cartesian product graph
\[
X=C_N^k\square G. 
\]
For a connected graph $Y$, we denote the number of spanning trees of $Y$ by $\tau(Y)$. 
We shall use the following standard facts. 
They are stated here in the form needed below.
\begin{theorem}[\cite{Kirchhoff1847}, Matrix--Tree Theorem]
\label{thmMatrixTree}
Let $Y$ be a connected graph with $n$ vertices, and let
\[
0=\lambda_1<\lambda_2\le\cdots\le\lambda_n
\]
be the Laplacian eigenvalues of $Y$.
Then
\[
\tau(Y)
=
\frac1n\prod_{q=2}^{n}\lambda_q.
\]
\end{theorem}

\begin{theorem}[\cite{ChebotarevShamis1997}, Matrix-forest theorem]
\label{thmForestResistance}
Let $Y$ be a connected graph.
For two distinct vertices $x,y\in V(Y)$, let $F_Y(x\mid y)$ be the number
of two-component spanning forests in which $x$ and $y$ belong to different
components.
Then
\[
F_Y(x\mid y)
=
\tau(Y)R_Y(x,y).
\]
\end{theorem}

\begin{theorem}[\cite{NashWilliams1959}, Commute time identity]
\label{thmCommuteTime}
Let $Y$ be a connected graph.
Then, for two vertices $x,y\in V(Y)$,
\[
h_Y(x,y)+h_Y(y,x)
=
2|E(Y)|R_Y(x,y).
\]
\end{theorem}

In what follows, put
\[
\theta_j=\frac{2\pi j}{N}, \qquad
\kappa_j
=2k-2\sum_{s=1}^{k}\cos(s\theta_j)
=2k-\mu_j.
\]
Then
\[
\kappa_0=0, \qquad
\kappa_j>0 \quad (j=1,\ldots,N-1).
\]

Let the distinct Laplacian eigenvalues of $G$ be
\[
0=\nu_0<\nu_1<\cdots<\nu_t,
\]
and let $m_\alpha$ be the multiplicity of $\nu_\alpha$. 

\subsection{Spanning trees}
The Laplacian eigenvalues of $C_N^k$ are
\[
\kappa_j \qquad (j=0, 1, \ldots, N-1).
\]
By the standard spectral formula for Cartesian product graphs
\cite{BrouwerHaemers2012}, 
the Laplacian spectrum of $X$, counted with multiplicities, consists of
\[
\kappa_j+\nu_\alpha, \qquad
0\le j\le N-1, \qquad
0\le\alpha\le t.
\]
Here, each pair $(j,\alpha)$ contributes with multiplicity $m_\alpha$, 
and $\kappa_0+\nu_0=0$ is the unique zero eigenvalue.

\begin{proposition}
\label{prop19}
Let $X=C_N^k\square G$.
Then
\[
\tau(X)
=\frac1{Nm}\prod_{j=1}^{N-1}\kappa_j\prod_{\alpha=1}^{t}\prod_{j=0}^{N-1}\left(\kappa_j+\nu_\alpha\right)^{m_\alpha}.
\]

Moreover,
\[
\tau(C_N^k\square G)
=\frac{\tau(C_N^k)}{m}\prod_{\alpha=1}^{t}\prod_{j=0}^{N-1}\left(\kappa_j+\nu_\alpha\right)^{m_\alpha}, 
\]
and
\[
\tau(C_N^k\square G)
=\tau(C_N^k)\tau(G)\prod_{\alpha=1}^{t}\prod_{j=1}^{N-1}\left(\kappa_j+\nu_\alpha\right)^{m_\alpha}.
\]
\end{proposition}

\begin{proof}
By Theorem \ref{thmMatrixTree}, the number of spanning trees of $X$ is obtained
from the product of the nonzero Laplacian eigenvalues of $X$.

The number of vertices of $X=C_N^k\square G$ is
\[
|V(X)|=Nm.
\]
Moreover, the Laplacian eigenvalues of $X$, counted with multiplicities, are
\[
\kappa_j+\nu_\alpha, \qquad 
0\le j\le N-1, \qquad
0\le\alpha\le t.
\]

Since $\kappa_0=0$ and $\nu_0=0$,
\[
\kappa_0+\nu_0=0
\]
is the unique zero eigenvalue. 

Therefore, the nonzero Laplacian eigenvalues of $X$ can be divided into the following two types. 

First, when $\alpha=0$, we have 
\[
\kappa_j+\nu_0=\kappa_j
\]
for $j=1,\ldots,N-1$. 

Next, when $\alpha=1,\ldots,t$, the eigenvalues
\[
\kappa_j+\nu_\alpha
\]
appear for $j=0,\ldots,N-1$, each with multiplicity $m_\alpha$. 

Hence, Theorem \ref{thmMatrixTree} gives
\[
\tau(X)
=\frac1{Nm}\prod_{j=1}^{N-1}\kappa_j\prod_{\alpha=1}^{t}\prod_{j=0}^{N-1}\left(\kappa_j+\nu_\alpha\right)^{m_\alpha}.
\]
This proves the first formula. 

Next, the Laplacian eigenvalues of $C_N^k$ are
\[
0=\kappa_0, \qquad
\kappa_1,\ldots,\kappa_{N-1}.
\]
Therefore, applying Theorem \ref{thmMatrixTree} again, we obtain
\[
\tau(C_N^k)
=\frac1N\prod_{j=1}^{N-1}\kappa_j.
\]
Thus,
\[
\prod_{j=1}^{N-1}\kappa_j
=N\tau(C_N^k).
\]
Substituting this identity into the first formula, we obtain
\[
\begin{aligned}
\tau(C_N^k\square G)
&=\frac1{Nm}N\tau(C_N^k)\prod_{\alpha=1}^{t}\prod_{j=0}^{N-1}\left(\kappa_j+\nu_\alpha\right)^{m_\alpha} \\
&=\frac{\tau(C_N^k)}{m}\prod_{\alpha=1}^{t}\prod_{j=0}^{N-1}\left(\kappa_j+\nu_\alpha\right)^{m_\alpha}.
\end{aligned}
\]
This proves the second formula.

Finally, we separate the terms with $j=0$ from the product appearing in the second formula. 
Since $\kappa_0=0$, we have
\[
\begin{aligned}
&\prod_{\alpha=1}^{t}\prod_{j=0}^{N-1}\left(\kappa_j+\nu_\alpha\right)^{m_\alpha} \\
&=\prod_{\alpha=1}^{t}\nu_\alpha^{m_\alpha}\prod_{\alpha=1}^{t}\prod_{j=1}^{N-1}\left(\kappa_j+\nu_\alpha\right)^{m_\alpha}.
\end{aligned}
\]

On the other hand, the nonzero Laplacian eigenvalues of $G$ are $\nu_\alpha$ with multiplicity $m_\alpha$
for each $\alpha=1,\ldots,t$. 
Therefore, Theorem \ref{thmMatrixTree} gives
\[
\tau(G)
=\frac1m\prod_{\alpha=1}^{t}\nu_\alpha^{m_\alpha}. 
\]
Thus,
\[
\prod_{\alpha=1}^{t}\nu_\alpha^{m_\alpha}
=m\tau(G).
\]

Substituting these identities into the second formula, we obtain
\[
\begin{aligned}
\tau(C_N^k\square G)
&=\frac{\tau(C_N^k)}{m} \cdot m\tau(G)\prod_{\alpha=1}^{t}\prod_{j=1}^{N-1}\left(\kappa_j+\nu_\alpha\right)^{m_\alpha} \\
&=\tau(C_N^k)\tau(G)\prod_{\alpha=1}^{t}\prod_{j=1}^{N-1}\left(\kappa_j+\nu_\alpha\right)^{m_\alpha}.
\end{aligned}
\]
This proves the third formula. 
\end{proof}

Using the polynomial defined in the preceding section,
\[
D_{k,\nu}(x)
=k+\frac{\nu}{2}-\sum_{s=1}^{k}T_s\left(\frac{x}{2}\right),
\]
we obtain the following representation. 

\begin{corollary}
\label{cor20}
We have
\[
\tau(C_N^k\square G)
=
\frac{\tau(C_N^k)}{m}
\prod_{\alpha=1}^{t}
\left(
2^N\prod_{j=0}^{N-1}D_{k,\nu_\alpha}(2\cos\theta_j)
\right)^{m_\alpha}. 
\]
\end{corollary}

\begin{proof}
It is sufficient to substitute
\[
2D_{k,\nu_\alpha}(2\cos\theta_j)
=\kappa_j+\nu_\alpha
\]
into Proposition \ref{prop19}. 
\end{proof}

\subsection{Two-component spanning forests}
For two distinct vertices $x, y \in V(X)$, 
let $F_X(x\mid y)$ denote the number of two-component spanning forests in which $x$ and $y$ belong to different connected components. 

\begin{proposition}
\label{prop21}
For two distinct vertices $x,y\in V(X)$, we have
\[
F_X(x\mid y)
=\frac{\tau(X)}{Nm(2k+r)}\left\{h_X(x,y)+h_X(y,x)\right\}.
\]
\end{proposition}

\begin{proof}
By Theorem \ref{thmForestResistance}, we have
\[
F_X(x\mid y)
=\tau(X)R_X(x,y).
\]

Since $X$ is $(2k+r)$-regular, we have
\[
|E(X)|
=\frac{Nm(2k+r)}2.
\]
By Theorem \ref{thmCommuteTime}, 
\[
h_X(x,y)+h_X(y,x)
=Nm(2k+r)R_X(x,y).
\]
Combining these identities proves the assertion.
\end{proof}

In general, average hitting times are not necessarily symmetric. 
However, for two vertices having the same $G$-coordinate, 
they are symmetric because of the symmetry on the $C_N^k$-side. 

\begin{corollary}
\label{cor22}
Let $a\in V(G)$ and $1 \le \ell \le N-1$. 
Then
\[
h_X((0,a),(\ell,a))
=h_X((\ell,a),(0,a)), 
\]
and
\[
F_X((0,a)\mid(\ell,a))
=\frac{2\tau(X)}{Nm(2k+r)}h_X((0,a),(\ell,a)).
\]
\end{corollary}

\begin{proof}
The map
\[
\sigma(x,c)=(\ell-x,c)
\]
is an automorphism of $X$ satisfying
\[
\sigma(0,a)=(\ell,a), \qquad
\sigma(\ell,a)=(0,a). 
\]
Therefore, the average hitting times are symmetric. 
Substituting this identity into Proposition \ref{prop21} proves the assertion. 
\end{proof}

Using Theorem \ref{thm16}, second-order linear recurrence sequences also appear in the number of two-component spanning forests. 

Substituting Theorem \ref{thm16} into Corollary \ref{cor22}, 
we obtain the following result. 

\begin{corollary}
\label{cor23}
For each $\alpha=1, \ldots, t$, assume that all roots of $D_{k,\nu_\alpha}(x)$ are simple. 
Then, for $a \in V(G)$ and $1\le \ell \le N-1$, we have
\[
\begin{aligned}
F_X((0,a)\mid(\ell,a))
&=\frac{\tau(X)}{Nmk}h_{C_N^k}(0,\ell) \\
&\quad-\tau(X)\sum_{\alpha=1}^{t}E_\alpha(a,a)\sum_{c=1}^{k}\frac1{\gamma_{\alpha,c}D_{k,\nu_\alpha}'(\phi_{\alpha,c})}\frac{V_\ell^{(\alpha,c)}V_{N-\ell}^{(\alpha,c)}}{V_N^{(\alpha,c)}}. 
\end{aligned}
\]
\end{corollary}

\begin{proof}
By Corollary \ref{cor22}, we have
\[
F_X((0,a)\mid(\ell,a))
=
\frac{2\tau(X)}{Nm(2k+r)}
h_X((0,a),(\ell,a)).
\]
Substituting the formula in Theorem \ref{thm16} into the right-hand side, 
we obtain
\[
\frac{2\tau(X)}{Nm(2k+r)}
\cdot
\frac{2k+r}{2k}
=
\frac{\tau(X)}{Nmk},
\]
and
\[
\frac{2\tau(X)}{Nm(2k+r)}
\cdot
\left(
-\frac{Nm(2k+r)}{2}
\right)
=
-\tau(X).
\]
Therefore, the desired formula follows.
\end{proof}

It is known that average hitting times on walk-regular graphs are symmetric
\cite{Georgakopoulos2012}.
In the present product setting, we verify the required symmetry directly from
Proposition \ref{prop04}.

\begin{corollary}
\label{cor24}
Suppose that $G$ is a walk-regular graph. 
Let $a, b \in V(G)$ and $0 \le \ell \le N-1$, and assume that $(0,a) \neq (\ell,b)$. 
Then
\[
h_X((0,a),(\ell,b))
=h_X((\ell,b),(0,a)). 
\]
Therefore,
\[
F_X((0,a)\mid(\ell,b))
=\frac{2\tau(X)}{Nm(2k+r)}h_X((0,a),(\ell,b)).
\]
\end{corollary}

\begin{proof}
By Lemma \ref{lemWalkRegular}, 
\[
E_\alpha(a,a)
=E_\alpha(b,b)
=\frac{m_\alpha}{m}. 
\]

On the other hand, by translation on the $C_N^k$-side, 
\[
h_X((\ell,b),(0,a))
=h_X((0,b),(N-\ell,a)).
\]
Using
\[
\cos((N-\ell)\theta_j)
=\cos(\ell\theta_j), \qquad
E_\alpha(b,a)=E_\alpha(a,b)
\]
in Proposition \ref{prop04}, 
the difference between the two average hitting times is determined only by the terms involving $E_\alpha(b,b)-E_\alpha(a,a)$. 
By walk-regularity, this difference is zero, 
and hence the average hitting times are symmetric. 
Applying Proposition \ref{prop21}, 
we obtain the formula for the number of two-component spanning forests. 
\end{proof}

\begin{remark}
\label{rem25}
The number of spanning trees is determined only by the Laplacian spectra of $C_N^k$ and $G$. 
On the other hand, the number of two-component spanning forests generally also involves information on the spectral projections
$E_\alpha(a,b)$. 

However, when the two vertices have the same $G$-coordinate, 
the second-order linear recurrence structure in Theorem \ref{thm16} is also reflected in the number of two-component spanning forests through Corollary \ref{cor23}. 
Furthermore, when $G$ is walk-regular, this formula can be described only in terms of the Laplacian eigenvalues of $G$ and their multiplicities. 
\end{remark}

\section{Explicit examples}
\label{sec06}
We first recall a standard fact that strongly regular graphs are walk-regular.

\begin{lemma}[\cite{GodsilMcKay1980}, Section 4]
\label{lemSRGWalkRegular}
Every strongly regular graph is walk-regular.
\end{lemma}

\begin{proof}
For the reader's convenience, we give the proof.

Let $G$ be a strongly regular graph with adjacency matrix $A_G$.
Then there exist constants $\lambda$ and $\mu$ such that
\[
A_G^2
=
rI+\lambda A_G+\mu(J-I-A_G).
\]
Therefore, $A_G^2$ belongs to the linear span of
\[
I,\quad A_G,\quad J.
\]
Since $G$ is regular, we also have
\[
A_GJ=JA_G=rJ.
\]
It follows by induction that every power $A_G^q$ belongs to the linear span of
\[
I,\quad A_G,\quad J.
\]
Since the diagonal entries of $I$, $A_G$, and $J$ are independent of the
vertex, the diagonal entries of $A_G^q$ are independent of the vertex.
Hence $G$ is walk-regular.
\end{proof}

In this section, we apply the formulas obtained in the preceding sections
to complete graphs and strongly regular graphs.
Complete graphs are vertex-transitive, and hence walk-regular.
Moreover, by Lemma \ref{lemSRGWalkRegular}, strongly regular graphs are also
walk-regular.
Therefore, Corollary \ref{corWalkRegular} can be applied to the average
hitting time between two vertices having the same $G$-coordinate.

We denote the Laplacian spectrum of a graph $G$ by
$\operatorname{Spec}_L(G)$, and write $\lambda^{(q)}$ to indicate that the
Laplacian eigenvalue $\lambda$ has multiplicity $q$.
Here, the number in the superscript parentheses represents the multiplicity,
not a power.

For $\nu>0$, let
\[
D_{k,\nu}(x)
=k+\frac{\nu}{2}-\sum_{s=1}^{k}T_s\left(\frac{x}{2}\right). 
\]
When all of its roots are simple, 
we denote them by $\phi_{\nu,1},\ldots,\phi_{\nu,k}$, and put
\[
\rho_{\nu,c}+\rho_{\nu,c}^{-1}=\phi_{\nu,c}, \qquad 
\rho_{\nu,c}=\eta_{\nu,c}^{\,2}, \qquad
\gamma_{\nu,c}=\eta_{\nu,c}+\eta_{\nu,c}^{-1}.
\]

Furthermore, define the sequence $\{V_n^{(\nu,c)}\}_{n\ge0}$ by
\[
V_0^{(\nu,c)}=0, \qquad
V_1^{(\nu,c)}=1, \qquad
V_{n+2}^{(\nu,c)}=\gamma_{\nu,c}V_{n+1}^{(\nu,c)}-V_n^{(\nu,c)},
\]
and write
\[
\mathcal A_{k,\nu}^{(N)}(\ell)
=\sum_{c=1}^{k}\frac1{\gamma_{\nu,c}D_{k,\nu}'(\phi_{\nu,c})}\frac{V_\ell^{(\nu,c)}V_{N-\ell}^{(\nu,c)}}{
V_N^{(\nu,c)}}.
\]
Whenever this notation is used, 
we assume that all roots of $D_{k,\nu}(x)$ are simple. 

Let $G$ be a connected $r$-regular walk-regular graph on $m$ vertices.
Let the distinct nonzero Laplacian eigenvalues of $G$ be
\[
\nu_1,\ldots,\nu_t,
\]
with multiplicities $m_1,\ldots,m_t$. 
Then Corollary \ref{corWalkRegular} becomes
\begin{equation}
\label{eq:walk-regular-general}
\begin{aligned}
h_{C_N^k\square G}((0,a),(\ell,a))
&=\frac{2k+r}{2k}h_{C_N^k}(0,\ell) \\
&\quad-\frac{N(2k+r)}{2}\sum_{\alpha=1}^{t}m_\alpha\mathcal A_{k,\nu_\alpha}^{(N)}(\ell).
\end{aligned}
\end{equation}

\subsection{The case $k=1$}
When $k=1$,
\[
D_{1,\nu}(x)
=\frac{\nu+2-x}{2},
\]
and hence
\[
\phi_\nu=\nu+2, \qquad
D_{1,\nu}'(\phi_\nu)=-\frac12.
\]
Moreover, if we put
\[
\rho_\nu
=\frac{\nu+2+\sqrt{\nu(\nu+4)}}{2},
\]
then
\[
\rho_\nu+\rho_\nu^{-1}
=\nu+2.
\]
Furthermore, we take
\[
\gamma_\nu=\sqrt{\nu+4},
\]
and define the sequence $\{V_n^{(\nu)}\}_{n\ge0}$ by
\[
V_0^{(\nu)}=0, \qquad
V_1^{(\nu)}=1, \qquad
V_{n+2}^{(\nu)}=\sqrt{\nu+4}\,V_{n+1}^{(\nu)}-V_n^{(\nu)}. 
\]

\begin{lemma}
\label{lemChebyshevProduct}
Let $\rho\ne0$, and put
\[
a=\rho+\rho^{-1}.
\]
Then
\[
\prod_{j=0}^{N-1}\left(a-2\cos\frac{2\pi j}{N}\right)
=\rho^N+\rho^{-N}-2.
\]
\end{lemma}

\begin{proof}
Put $\omega=e^{2\pi i/N}$. 
Then
\[
2\cos\frac{2\pi j}{N}
=\omega^j+\omega^{-j}.
\]

Moreover, since 
\[
a=\rho+\rho^{-1},
\]
we have
\[
\begin{aligned}
a-\omega^j-\omega^{-j}
&=\rho+\rho^{-1}-\omega^j-\omega^{-j} \\
&=\frac{(\rho-\omega^j)(\rho-\omega^{-j})}{\rho}.
\end{aligned}
\]
Therefore,
\[
\begin{aligned}
\prod_{j=0}^{N-1}\left(a-2\cos\frac{2\pi j}{N}\right)
&=\rho^{-N}\prod_{j=0}^{N-1}(\rho-\omega^j)\prod_{j=0}^{N-1}(\rho-\omega^{-j}).
\end{aligned}
\]

The numbers $\omega^0,\omega^1,\ldots,\omega^{N-1}$ are all the $N$-th roots of unity, 
and $\omega^0,\omega^{-1},\ldots,\omega^{-(N-1)}$ form the same set. 
Therefore, 
\[
\prod_{j=0}^{N-1}(\rho-\omega^j)
=\rho^N-1,
\]
and
\[
\prod_{j=0}^{N-1}(\rho-\omega^{-j})
=\rho^N-1.
\]
Hence,
\[
\begin{aligned}
\prod_{j=0}^{N-1}\left(a-2\cos\frac{2\pi j}{N}\right)
&=\rho^{-N}(\rho^N-1)^2 \\
&=\rho^N+\rho^{-N}-2.
\end{aligned}
\]
\end{proof}


\begin{proposition}
\label{propK1WalkRegular}
Let $G$ be a connected $r$-regular walk-regular graph on $m$ vertices. 
Let its distinct nonzero Laplacian eigenvalues be $\nu_1, \ldots, \nu_t$, 
with multiplicities $m_1,\ldots,m_t$. 

Then, for any $a \in V(G)$ and $0 \le \ell \le N-1$, we have
\[
\begin{aligned}
h_{C_N\square G}((0,a),(\ell,a))
&=\frac{r+2}{2}\,\ell(N-\ell) \\
&\quad+N(r+2)\sum_{\alpha=1}^{t}\frac{m_\alpha}{\sqrt{\nu_\alpha+4}}\frac{V_\ell^{(\nu_\alpha)}V_{N-\ell}^{(\nu_\alpha)}}{
V_N^{(\nu_\alpha)}}.
\end{aligned}
\]
Moreover,
\[
\tau(C_N\square G)
=\frac{N}{m}\prod_{\alpha=1}^{t}\left(\rho_{\nu_\alpha}^{N}+\rho_{\nu_\alpha}^{-N}-2\right)^{m_\alpha}.
\]
\end{proposition}

\begin{proof}
Substitute $k=1$ into Eq. \eqref{eq:walk-regular-general}. 
In this case, 
\[
h_{C_N}(0,\ell)=\ell(N-\ell),
\]
and
\[
\mathcal A_{1,\nu}^{(N)}(\ell)
=-\frac{2}{\sqrt{\nu+4}}\frac{V_\ell^{(\nu)}V_{N-\ell}^{(\nu)}}{V_N^{(\nu)}}. 
\]
Therefore, we obtain the formula for the average hitting time. 

Moreover, by Proposition \ref{prop19} and Lemma \ref{lemChebyshevProduct}, 
\[
\begin{aligned}
\tau(C_N\square G)
&=\frac{N}{m}\prod_{\alpha=1}^{t}\prod_{j=0}^{N-1}\left(\nu_\alpha+2-2\cos\theta_j\right)^{m_\alpha} \\
&=\frac{N}{m}\prod_{\alpha=1}^{t}\left(\rho_{\nu_\alpha}^{N}+\rho_{\nu_\alpha}^{-N}-2\right)^{m_\alpha}.
\end{aligned}
\]
\end{proof}

\begin{remark}
\label{remK1}
When $k=1$, one recurrence relation
\[
V_{n+2}^{(\nu_\alpha)}
=\sqrt{\nu_\alpha+4}\,V_{n+1}^{(\nu_\alpha)}-V_n^{(\nu_\alpha)}
\]
corresponds to each nonzero Laplacian eigenvalue $\nu_\alpha$ of $G$. 
\end{remark}

\subsection{Complete graphs}
Let $m\ge2$, and let $G=K_m$. 
The graph $K_m$ is $(m-1)$-regular. 
By the standard spectrum of complete graphs
\cite{BrouwerHaemers2012}, its Laplacian spectrum is
\[
\operatorname{Spec}_L(K_m)
=\left\{0^{(1)},m^{(m-1)}\right\}.
\]

\begin{proposition}
\label{propCompleteGeneral}
Assume that all roots of $D_{k,m}(x)$ are simple, 
and denote them by $\phi_{m,1},\ldots,\phi_{m,k}$. 
Moreover, put
\[
\mathcal G_{m,c}(\ell)
=\frac{N\left(\rho_{m,c}^{\ell}+\rho_{m,c}^{N-\ell}\right)}{\left(\rho_{m,c}-\rho_{m,c}^{-1}\right)\left(\rho_{m,c}^{N}-1\right)}.
\]

Then, for $a,b\in V(K_m)$ and $0\le\ell\le N-1$, we have
\[
\begin{aligned}
h_{C_N^k\square K_m}((0,a),(\ell,b))
&=\frac{2k+m-1}{2k}h_{C_N^k}(0,\ell)\\
&\quad-\frac{2k+m-1}{2}
\sum_{c=1}^{k}
\frac{
(m-1)\mathcal G_{m,c}(0)
-(m\delta_{ab}-1)\mathcal G_{m,c}(\ell)
}{
D_{k,m}'(\phi_{m,c})
}.
\end{aligned}
\]
In particular, when $a=b$,
\[
\begin{aligned}
h_{C_N^k\square K_m}((0,a),(\ell,a))
&=\frac{2k+m-1}{2k}h_{C_N^k}(0,\ell) \\
&\quad-\frac{N(2k+m-1)(m-1)}{2}\mathcal A_{k,m}^{(N)}(\ell).
\end{aligned}
\]
\end{proposition}

\begin{proof}
The orthogonal projection onto the nonzero Laplacian eigenspace is
\[
E_1=I_m-\frac1mJ_m,
\]
and
\[
mE_1(b,b)=m-1, \qquad
mE_1(a,b)=m\delta_{ab}-1. 
\]
Substituting these identities into Proposition \ref{prop12} gives the first
formula.

When $a=b$, the second formula follows from Corollary \ref{corWalkRegular}
with the single nonzero Laplacian eigenvalue $m$ of multiplicity $m-1$.
\end{proof}

This gives the formula when the second factor is a complete graph. 
In particular, taking $m=2$ and $k=1$, we obtain the following result. 

\begin{corollary}
\label{corLadder}
Let $X=C_N\square K_2$, and let
\[
V(K_2)=\{0,1\}.
\]
Then, for $0\le\ell\le N-1$ and $s\in\{0,1\}$, we have
\[
\begin{aligned}
h_X((0,0),(\ell,s))
&=\frac32\,\ell(N-\ell) \\
&\quad+\frac{\sqrt3\,N}{2}\frac{1+\rho^N-(-1)^s\left(\rho^\ell+\rho^{N-\ell}\right)}{\rho^N-1},
\end{aligned}
\]
where $\rho=2+\sqrt3$. 
\end{corollary}

\begin{proof}
We have
\[
D_{1,2}(x)=2-\frac{x}{2}, \qquad
\phi=4, \qquad
D_{1,2}'(4)=-\frac12.
\]
Moreover, the solution $\rho>1$ satisfying
\[
\rho+\rho^{-1}=4, \qquad
\rho-\rho^{-1}=2\sqrt3
\]
is
$\rho=2+\sqrt3$. 
Substituting these identities and
\[
2\delta_{0s}-1=(-1)^s
\]
into Proposition \ref{propCompleteGeneral} proves the assertion. 
\end{proof}

\subsection{Strongly regular graphs}
Let $G$ be a nontrivial strongly regular graph with parameters
\[
(m,r,\lambda,\mu).
\]
Thus, $G$ has $m$ vertices, is $r$-regular, 
any two adjacent vertices have exactly $\lambda$ common neighbors, 
and any two non-adjacent vertices have exactly $\mu$ common neighbors. 
We use the standard spectral properties of strongly regular graphs; see
\cite{BrouwerHaemers2012}. 
Let the eigenvalues of its adjacency matrix be
\[
r,\qquad
\vartheta,\qquad
\zeta.
\]
If the multiplicities of $\vartheta$ and $\zeta$ are $f$ and $g$, respectively, then
\[
\operatorname{Spec}_L(G)
=\left\{0^{(1)},(r-\vartheta)^{(f)},(r-\zeta)^{(g)}\right\}.
\]
In what follows, we write
\[
\nu_\vartheta=r-\vartheta, \qquad
\nu_\zeta=r-\zeta. 
\]

\begin{proposition}
\label{propSRG}
Assume that all roots of $D_{k,\nu_\vartheta}(x)$ and $D_{k,\nu_\zeta}(x)$ are simple. 
Then, for $a\in V(G)$ and $0\le\ell\le N-1$, we have 
\[
\begin{aligned}
h_{C_N^k\square G}((0,a),(\ell,a))
&=\frac{2k+r}{2k}h_{C_N^k}(0,\ell) \\
&\quad-\frac{N(2k+r)}{2}\left(f\mathcal A_{k,\nu_\vartheta}^{(N)}(\ell)+g\mathcal A_{k,\nu_\zeta}^{(N)}(\ell)\right).
\end{aligned}
\]

Moreover, without the assumption that the roots are simple, we have
\[
\tau(C_N^k\square G)
=\frac{\tau(C_N^k)}{m}\prod_{j=0}^{N-1}\left(\kappa_j+\nu_\vartheta\right)^f\left(\kappa_j+\nu_\zeta\right)^g.
\]
\end{proposition}

\begin{proof}
By Lemma \ref{lemSRGWalkRegular}, the graph $G$ is walk-regular.
The nonzero Laplacian eigenvalues of $G$ are
\[
\nu_\vartheta=r-\vartheta,
\qquad
\nu_\zeta=r-\zeta,
\]
with multiplicities $f$ and $g$, respectively.
Substituting these data into Equation \eqref{eq:walk-regular-general} gives
the formula for the average hitting time.

The formula for the number of spanning trees follows by substituting the
same Laplacian spectrum into Proposition \ref{prop19}.
This part does not require the simplicity assumption on the roots.
\end{proof}

\subsubsection{Complete bipartite graphs}
Let $s \ge 2$, and let $G=K_{s,s}$.
Then $G$ is a strongly regular graph with parameters
\[
(2s,s,0,s).
\]
By the standard spectrum of complete bipartite graphs
\cite{BrouwerHaemers2012}, its Laplacian spectrum is
\[
\operatorname{Spec}_L(K_{s,s})
=\left\{0^{(1)},s^{(2s-2)},(2s)^{(1)}\right\}.
\]
That is, the multiplicities of the Laplacian eigenvalues
$0$, $s$, and $2s$ are $1$, $2s-2$, and $1$, respectively.

\begin{corollary}
\label{corCompleteBipartite}
Assume that all roots of $D_{k,s}(x)$ and $D_{k,2s}(x)$ are simple. 
Then, for $a\in V(K_{s,s})$ and $0\le\ell\le N-1$, we have
\[
\begin{aligned}
h_{C_N^k\square K_{s,s}}((0,a),(\ell,a))
&=\frac{2k+s}{2k}h_{C_N^k}(0,\ell) \\
&\quad-\frac{N(2k+s)}{2}\left((2s-2)\mathcal A_{k,s}^{(N)}(\ell)+\mathcal A_{k,2s}^{(N)}(\ell)\right).
\end{aligned}
\]

In particular, when $k=1$,
\[
\begin{aligned}
h_{C_N\square K_{s,s}}((0,a),(\ell,a))
&=\frac{s+2}{2}\,\ell(N-\ell) \\
&\quad+N(s+2)
\left(\frac{2s-2}{\sqrt{s+4}}\frac{V_\ell^{(s)}V_{N-\ell}^{(s)}}{V_N^{(s)}}+\frac1{\sqrt{2s+4}}\frac{V_\ell^{(2s)}V_{N-\ell}^{(2s)}}{V_N^{(2s)}}
\right).
\end{aligned}
\]
\end{corollary}

\begin{proof}
It is sufficient to substitute the Laplacian spectrum of $K_{s,s}$ into Proposition \ref{propSRG} and Proposition \ref{propK1WalkRegular}. 
\end{proof}


\subsubsection{The Petersen graph}
Let $\mathcal P$ denote the Petersen graph. 
It is a strongly regular graph with parameters $(10,3,0,1)$. 
By the standard spectrum of the Petersen graph
\cite{BrouwerHaemers2012}, its Laplacian spectrum is
\[
\operatorname{Spec}_L(\mathcal P)
=\left\{0^{(1)},2^{(5)},5^{(4)}\right\}.
\]
That is, the multiplicities of the Laplacian eigenvalues
$0$, $2$, and $5$ are $1$, $5$, and $4$, respectively. 

\begin{corollary}
\label{corPetersen}
Assume that all roots of $D_{k,2}(x)$ and $D_{k,5}(x)$ are simple. 
Then, for $a\in V(\mathcal P)$ and $0\le\ell\le N-1$, we have
\[
\begin{aligned}
h_{C_N^k\square\mathcal P}((0,a),(\ell,a))
&=\frac{2k+3}{2k}h_{C_N^k}(0,\ell) \\
&\quad-\frac{N(2k+3)}{2}\left(5\mathcal A_{k,2}^{(N)}(\ell)+4\mathcal A_{k,5}^{(N)}(\ell)\right).
\end{aligned}
\]

In particular, when $k=1$, 
\[
\begin{aligned}
h_{C_N\square\mathcal P}((0,a),(\ell,a))
&=\frac52\,\ell(N-\ell) \\
&\quad+5N
\left(\frac5{\sqrt6}\frac{V_\ell^{(2)}V_{N-\ell}^{(2)}}{V_N^{(2)}}+\frac43\frac{V_\ell^{(5)}V_{N-\ell}^{(5)}}{V_N^{(5)}}\right),
\end{aligned}
\]
where
\[
V_{n+2}^{(2)}=\sqrt6\,V_{n+1}^{(2)}-V_n^{(2)}, \qquad
V_{n+2}^{(5)}=3V_{n+1}^{(5)}-V_n^{(5)}.
\]
\end{corollary}

\begin{proof}
It is sufficient to substitute the Laplacian spectrum of $\mathcal P$ into Proposition \ref{propSRG} and Proposition \ref{propK1WalkRegular}. 
\end{proof}

\section{Conclusion and future problems}
\label{sec07}

In this paper, for a connected $r$-regular graph $G$ on $m$ vertices, 
we investigated the average hitting times of the simple random walk on the Cartesian product graph $C_N^k\square G$. 

By combining discrete Fourier analysis in the $C_N^k$ direction with the Laplacian spectral decomposition of $G$, 
we decomposed the average hitting time into a component proportional to the average hitting time on $C_N^k$ and correction terms arising from the nonzero Laplacian eigenspaces of $G$. 
Furthermore, for each nonzero Laplacian eigenvalue, 
we introduced a Chebyshev-type polynomial, 
and when all of its roots are simple, 
we expressed the correction term as a finite Green-type sum. 
In particular, for two vertices having the same $G$-coordinate, 
we transformed the Green-type representation into a second-order linear recurrence representation of the form 
\[
\frac{V_\ell V_{N-\ell}}{V_N}. 
\]
This shows that the recurrence structures appearing on $C_N^k$ in our previous work are also preserved on the Cartesian product graph $C_N^k\square G$. 

Moreover, when $G$ is a walk-regular graph, 
the average hitting time between two vertices having the same $G$-coordinate can be described only in terms of the Laplacian eigenvalues of $G$ and their multiplicities. 
We also derived formulas for the number of spanning trees and the number of two-component spanning forests, 
and gave examples involving complete graphs, 
complete bipartite graphs, and the Petersen graph.

One possible direction for future research is to extend the present results to the case in which the second factor $G$ is a nonregular graph. 
In this case, since the stationary distribution is not uniform, 
a formulation based on the transition probability matrix or the normalized Laplacian will be required. 
It would also be interesting to investigate whether similar recurrence structures appear when $C_N^k$ is replaced by a general circulant graph or a Cayley graph.


\end{document}